\documentclass[a4paper,12pt]{amsart}

\usepackage{vmargin}
\usepackage[colorlinks=true,linkcolor=blue,citecolor=blue,urlcolor=blue]{hyperref}
\usepackage{bookmark}
\usepackage{amsthm,thmtools,amssymb,amsmath,amscd,amsfonts}
\usepackage{mathrsfs}
\usepackage{stmaryrd}

\usepackage[bibstyle=alphabetic,citestyle=alphabetic,backend=bibtex,maxbibnames=99]{biblatex}
\bibliography{Bibliography}

\usepackage{fancyhdr}
\usepackage{esint}

\usepackage{enumerate}

\usepackage{pictexwd,dcpic}

\usepackage{graphicx}
\usepackage[utf8]{inputenc}

\makeatletter
\def\subsection{\@startsection{subsection}{2}%
  \z@{.5\linespacing\@plus.7\linespacing}{.1\linespacing}%
  {\normalfont\bfseries}}
\makeatother

\declaretheorem[name=Theorem,numberwithin=section]{thm}
\declaretheorem[name=Remark,style=remark,sibling=thm]{rem}
\declaretheorem[name=Lemma,sibling=thm]{lemma}
\declaretheorem[name=Proposition,sibling=thm]{prop}
\declaretheorem[name=Conjecture,sibling=thm]{conj}

\declaretheorem[name=Example,style=remark,sibling=thm]{example}

\numberwithin{equation}{section}

\usepackage{cleveref}
\crefname{lemma}{Lemma}{Lemmata}
\crefname{prop}{Proposition}{Propositions}
\crefname{thm}{Theorem}{Theorems}
\crefname{cor}{Corollary}{Corollaries}
\crefname{conj}{Conjecture}{Conjectures}
\crefname{defn}{Definition}{Definitions}
\crefname{example}{Example}{Examples}
\crefname{rem}{Remark}{Remarks}
\crefname{ass}{Assumption}{Assumptions}
\crefname{not}{Notation}{Notation}
\crefname{section}{Section}{Sections}

\newcommand{\cn}{\colon}

\newcommand{\N}{\mathbb{N}}
\newcommand{\R}{\mathbb{R}}
\newcommand{\Z}{\mathbb{Z}}
\renewcommand{\S}{\mathbb{S}}
\renewcommand{\H}{\mathbb{H}}
\newcommand{\C}{\mathbb{C}}
\newcommand{\K}{\mathbb{K}}

\newcommand{\B}{\mathbb{B}}

\renewcommand{\a}{\alpha}
\renewcommand{\b}{\beta}
\newcommand{\g}{\gamma}
\renewcommand{\d}{\delta}
\newcommand{\e}{\epsilon}
\renewcommand{\k}{\kappa}
\renewcommand{\l}{\lambda}
\renewcommand{\o}{\omega}
\renewcommand{\t}{\theta}
\newcommand{\s}{\sigma}
\newcommand{\p}{\varphi}
\newcommand{\z}{\zeta}

\renewcommand{\O}{\Omega}
\newcommand{\D}{\Delta}
\newcommand{\G}{\Gamma}

\newcommand{\Lie}{\mathcal{L}}
\newcommand{\inpr}[2]{\left\langle #1,#2 \right\rangle}
\newcommand{\abs}[1]{\left\lvert{#1}\right\rvert}

\newcommand{\x}{\times}
\DeclareMathOperator{\Tr}{Tr}

\DeclareMathOperator{\dive}{div}

\DeclareMathOperator{\Ric}{Ric}

\DeclareMathOperator{\Rm}{Rm}
\DeclareMathOperator{\opRm}{\mathcal{R}m}
\DeclareMathOperator{\Sc}{R}
\DeclareMathOperator{\Ein}{E}
\DeclareMathOperator{\opEin}{\mathcal{E}}
\DeclareMathOperator{\Grav}{G}

\DeclareMathOperator{\W}{\mathcal{W}}

\DeclareMathOperator{\adj}{adj}
\DeclareMathOperator{\Sym}{Sym}

\newcommand{\eq}[1]{\begin{equation}\begin{alignedat}{2} #1 \end{alignedat}\end{equation}}

\newcommand{\br}[1]{\left(#1\right)}

\protected\def\ignorethis#1\endignorethis{}
\let\endignorethis\relax

\author[P. Bryan]{Paul Bryan}
\address{Department of Mathematics, Macquarie University
NSW 2109, Australia}
\email{\href{mailto:paul.bryan@uq.edu.au}{paul.bryan@uq.edu.au}}
\urladdr{\href{http://pabryan.github.io}{http://pabryan.github.io/}}

\author[M.N. Ivaki]{Mohammad N. Ivaki}
\address{Department of Mathematics, University of Toronto, Ontario,
M5S 2E4, Canada}
\email{\href{mailto:m.ivaki@utoronto.ca}{m.ivaki@utoronto.ca}}

\author[J. Scheuer]{Julian Scheuer}
\address{Department of Mathematics, Columbia University
New York, NY 10027, USA}
\email{\href{mailto:jss2291@columbia.edu}{jss2291@columbia.edu}}
\urladdr{\href{https://home.mathematik.uni-freiburg.de/scheuer/}{https://home.mathematik.uni-freiburg.de/scheuer/}}

\DeclareMathOperator{\Ob}{O}

\DeclareMathOperator{\T}{T}
\DeclareMathOperator{\Dv}{D}
\DeclareMathOperator{\xcf}{\sigma}
\DeclareMathOperator{\dtxcf}{\xcf_{\operatorname{DT}}}
\renewcommand{\L}{\ensuremath{\operatorname{L}}}
\DeclareMathOperator{\dtrf}{\Ric_{\operatorname{DT}}}

\begin{document}

\title[Negatively Curved Three Manifolds]{Negatively Curved Three-Manifolds, Hyperbolic Metrics, Isometric Embeddings In Minkowski Space And The Cross Curvature Flow}

\date{\today}

\subjclass[2010]{58J35, 35K10, 58B20}
\keywords{Negative curvature, embedding, Minkowski, space-like}

\begin{abstract}
This short note is a mostly expository article examining negatively curved three-manifolds. We look at some rigidity properties related to isometric embeddings into Minkowski space. We also review the Cross Curvature Flow (XCF) as a tool to study the space of negatively curved metrics on hyperbolic three-manifolds, the largest and least understood class of model geometries in Thurston's Geometrisation. The relationship between integrability and embedability yields interesting insights, and we show that solutions with fixed Einstein volume are precisely the integrable solutions, answering a question posed by Chow and Hamilton when they introduced the XCF.
\end{abstract}

\maketitle

\section{Introduction}
\label{sec:intro}

In the early 1980's, Thurston announced the Geometrisation Conjecture \cite{MR648524}. At around the same time, Hamilton introduced the Ricci flow \cite{Hamilton:/1982}. The Geometrisation Conjecture claimed that closed three-manifolds could be decomposed into pieces modelled on geometric structures of eight possible types, while the Ricci flow deformed metrics by their Ricci curvature. At that time, Thurston classified the possible geometric structures into the eight types and proved the Geometrisation Conjecture for Haken manifolds, while Hamilton obtained convergence to a constant sectional curvature metric provided the initial metric has positive Ricci curvature.

Needless to say, both these seminal works sparked off tremendous developments continuing to this day. A crowning achievement was Perelman's resolution of the Geometrisation Conjecture using the Ricci flow with surgery \cite{2003math......7245P,2003math......3109P,2002math.....11159P}. The legend goes that at the urging of Yau, Hamilton initiated a program to use the Ricci flow to prove the Geometrisation Conjecture. Roughly speaking, the Ricci flow  tends to smooth out irregularities in curvature, but singularities may occur. Hamilton outlined an approach to cut out the singularities with analytically controlled topological surgeries and continue the flow. Perelman proved that this process does indeed work, with only finitely many such surgeries required after which the remaining pieces converge to one of Thurston's eight model geometries. Tracing the process back provides the necessary decomposition to resolve the Geometrisation Conjecture.

Of the eight geometries, essentially only the hyperbolic geometries are not fully understood. All the other seven cases may be enumerated in a similar fashion to the uniformisation of surfaces. Around the time Perelman was completing his work on the Ricci flow, Chow and Hamilton introduced a new flow, the Cross Curvature Flow (XCF) \cite{MR2055396}. This flow deforms initially negatively curved metrics (also positive but we won't focus on that case) by a fully nonlinear parabolic equation. The Ricci flow is to the heat equation as the XCF is to a Monge-Amper\'e equation. The aim of the XCF is to deform negatively curved metrics to hyperbolic metrics thus illuminating the structure of the eighth and least understood model geometry, namely the hyperbolic geometries.

Here we will examine the status of this program, describing the known results to date and drawing together some observations made in the literature. We will touch on some of the difficulties faced in this program with the hope of reinvigorating the study of the XCF so that fresh insights may yield new results for this fascinating flow.

In \cite{MR2055396}, evidence that the XCF deforms arbitrary negatively curved metrics (after normalisation) to a hyperbolic metric was given. Further evidence was provided in \cite{MR2448593} showing that hyperbolics metrics are asymptotically stable. See \Cref{subsec:xcf_hyperbolic_convergence} below for details. The original short time existence proof based on the Nash-Moser implicit function theorem in \cite{MR2055396} was not complete and a complete proof based on the DeTurck trick was given in \cite{MR2207496}. Short time existence and uniqueness is discussed in \Cref{subsec:xcf_existence_uniqueness}. In the case the universal cover embeds isometrically into Minkowski space, the Gauss curvature flow is equivalent to the XCF (\Cref{lem:xcf_gcf})and convergence to the hyberbolic metric follows by \cite{MR3344442}. This is discussed in \Cref{sec:embed_intg}, where it is also shown that embedding is equivalent to an integrability condition (\Cref{thm:intg_embed}). Whether a Harnack inequality holds was raised in \cite{MR2055396}. The Harnack inequality for integrable solutions, as well as rigidity of solitons is discussed in \Cref{subsec:xcf_harnack_solitons}. Another question was to classify those solutions with constant Einstein volume which is answered in \Cref{subsec:xcf_volume}.

There are other results for the XCF contained in the literature that space precludes their discussion here. Briefly, long time existence and expansion to infinity of a square torus bundle is obtained in \cite{MR2222245}. Long time existence, convergence results and singularity analysis on locally homogeneous spaces of variable curvature where the XCF reduces to an ODE is examined in \cite{MR2407107,MR2653711,MR2426751} as well as the backwards behaviour in \cite{MR2601352}. On a solid torus, long time existence and curvature bounds for the XCF starting at the \(2\pi\)-metric of Gromov and Thurston is obtained in \cite{MR2602839}. Finally, uniqueness and backwards uniqueness are addressed in \cite{MR3575926,MR3544962}.

\section{Geometrisation Of Three Manifolds}
\label{sec:geometrisation}

For surfaces, the uniformisation theorem and Gauss-Bonnet theorem provide a complete picture of the topology and its relation to curvature. The three dimensional case is more complicated than the two dimensional case, but is almost completely understood thanks to Perelman's successful completion \cite{2003math......7245P,2003math......3109P,2002math.....11159P} of Hamilton's program based on the Ricci flow \cite{Hamilton:/1982} to solve the Poincar\'e and Thurston geometrisation conjectures \cite{MR648524}. The remaining piece of the puzzle is the structure of hyperbolic, closed three manifolds. We include here a brief description and refer the reader to \cite{MR705527} and \cite{MR1435975} for in depth discussions of geometrisation and \cite{MR3186136,MR2334563,MR2460872} for expositions of the Hamilton-Perelman proof. Unless explicitly stated otherwise, the results described here may be found in these references.

The geometrisation conjecture may be stated as follows:

\begin{thm}[Thurston Geometrisation]
\label{thm:geometrisation}

Every closed three manifold decomposes as a connected sum of prime manifolds, each of which may be cut along tori so that the interior of the resulting manifolds each admits a unique geometric structure of with finite volume from among a possible eight types.
\end{thm}

A prime manifold is simply a manifold that cannot be written as a non-trivial connected sum. The decomposition into prime manifolds was given in \cite{MR0142125}. To say that \(M\) admits a geometric structure is to say that \(M\) is diffeomorphic to \(X/\Gamma\) where \(X\) is a \(G\)-manifold for some Lie group \(G\) acting transitively on \(X\) with compact stabilisers, and \(\Gamma\) is a discrete subgroup of \(G\) acting freely on \(X\). The classification into eight types of finite volume geometric structures was given by Thurston. Finally the remaining part of the theorem, that such a decomposition of prime manifolds exists was proven by Hamilton and Perelman using the Ricci flow with surgery. The eight geometries are
\[
\R^3, \S^3, \H^3, \S^2 \times \R, \H^2 \times \R, \widetilde{SL}_2(\R), \operatorname{Nil}, \operatorname{Solv}.
\]

In the case of \(\operatorname{Solv}\) these are precisely torus and Klein bottle bundles over \(S^1\) or the union of two twisted \(1\)-bundles over the torus or Klein bottle. The remaining six non-hyperbolic geometries are all Seifert Fibre bundles, completely determined by the Euler characteristic of the base space, \(\chi\) and the Euler number of the bundle, \(e\).

The only remaining case then is the hyperbolic case and this has yet has no classification. As a consequence of geometrisation, we have

\begin{thm}[Hyperbolisation]
\label{thm:hyperbolisation}

A closed three-manifold admitting a metric of negative sectional curvature also admits a hyperbolic metric. That is, a metric of constant negative sectional curvature.
\end{thm}

Thurston proved this result for atoroidal Haken manifolds, and the general result is a consequence of geometrisation as proven by Hamilton and Perelman. The route to hyperbolisation via geometrisation is quite indirect since the Ricci flow does not generally preserve negative curvature in dimensions greater than two and furthermore, singularities may form along the Ricci flow so that surgery is necessary. The decomposition obtained in \Cref{thm:geometrisation} is not unique so that the Ricci flow with surgery need not produce the hyperbolic geometry in the limit. \Cref{thm:hyperbolisation} is deduced a posteriori from the geometrisation conjecture rather than as a direct consequence the Ricci flow.

Resolving \Cref{conj:chow_hamilton} by removing the integrability condition of \Cref{thm:intg_const_curv} would give a more direct proof of hyperbolisation, while strengthening the result to the statement that arbitrary negatively curved metrics are homotopic to a hyperbolic metric. Note that by the Mostow rigidity theorem \cite{MR0236383}, hyperbolic structures are classified by fundamental group and are essentially unique. That is, if the fundamental group \(\pi_1(M_1)\) of a closed hyperbolic manifold, \((M_1, g_1)\) is isomorphic to a \(\pi_1(M_2)\) for another hyperbolic manifold \((M_2, g_2)\), then in fact \((M_1, g_1)\) and \((M_2, g_2)\) are isometric. Equivalently, any homotopy equivalence of hyperbolic manifolds may in fact be homotopied to an isometry. Then \Cref{conj:chow_hamilton} would show that the space of negatively curved metrics on a hyperbolic three-manifold is contractible.

\section{Embeddability and hyperbolic metrics}
\label{sec:embed_intg}

Let \((M, g)\) be a compact, Riemannian manifold with strictly negative curvature. Let \(\pi\colon (\tilde{M}, \tilde{g}) \to (M, g)\) be the Riemannian universal cover so that \(\pi : \tilde{M} \to M\) is a covering map with \(\tilde{M}\) simply connected and \(\tilde{g} = \pi^{\ast} g\). Let \(G\) denote the deck transformation group of the cover and observe that \(\tilde{g}\) is invariant under \(G\). That is, \(G \leq \text{Diff}(\tilde{M})\) is a group of diffeomorphisms of \(\tilde{M}\) and \(\varphi^{\ast} \tilde{g} = \tilde{g}\) for all \(\varphi \in G\) so that \(G\) acts by isometry on \((\tilde{M}, \tilde{g})\). Then \(\tilde{g}\) induces a metric \(\bar{g}\) on the quotient \(\tilde{M}/G\) such that
\[
(\tilde{M}/G, \bar{g}) \underset{\simeq}{\to} (M, g)
\]
is an isometry and the quotient map \(\tilde{M} \to \tilde{M}/G\) is just \(\pi\) under this identification. Then \((\tilde{M}/G, \bar{g})\) is a compact Riemannian quotient and we say \((\tilde{M}, \tilde{g})\) is a co-compact Riemannian manifold.

Now, since \((M, g)\) has strictly negative sectional curvature, so does \((\tilde{M}, \tilde{g})\) hence by the Cartan-Hadamard theorem, \(\tilde{M} \simeq \R^3\) is diffeomorphic to \(\R^3\) via the exponential map. In particular we may equip \(\tilde{M}\) with the hyperbolic metric, \(\tilde{g}_{\H}\) of constant, negative sectional curvature equal to \(-1\). Let us write \(G_{\H}\) for the isometry group of \((\tilde{M}, \tilde{g}_{\H})\).

On \(\tilde{M}\), there is a simple, smooth homotopy from \(\tilde{g}\) to \(\tilde{g}_{\H}\):
\[
\tilde{h}(t) = t\tilde{g} + (1-t)\tilde{g}_{\H}, \quad t \in [0, 1].
\]
This gives rise to the following simple lemma:

\begin{lemma}
\label{lem:const_neg}

Let \((M, g)\) be a compact manifold of strictly negative sectional curvature. Then the following statements are equivalent:
\begin{enumerate}[(i)]
\item \label{enum:neg_met} \(M\) admits a metric of constant, negative sectional curvature.
\item \label{enum:deck_met} \(\tilde{g}_{\H}\) is invariant under \(G\).
\item \label{enum:subgroup} G is a subgroup of \(G_{\H}\).
\item \label{enum:homo_met} \(g\) is smoothly homotopic to a metric of constant, negative sectional curvature.
\item \label{enum:homo_deck} Every \(G\)-invariant metric \(\tilde{g}\) on \(\tilde{M}\) is smoothly homotopic to \(\tilde{g}_{\H}\) via a smooth \(G\)-invariant homotopy.
\end{enumerate}
\end{lemma}

\begin{proof}
(i)\(\ \Rightarrow\ \)(ii)
The pullback of a constant curvature metric under the Riemannian covering $\pi$ is the hyperbolic one, which is thus \(G\)-invariant.

(ii)\(\ \Rightarrow\ \)(iii)
Clear, since \(G_{\H}\) is the whole isometry group.

(iii)\(\ \Rightarrow\ \)(iv)
Since \(\tilde{g}\) and \(\tilde{g}_{\H}\) are invariant under G, so is \(\tilde{h}(t)\), which in turn descends to a homotopy on \(\tilde{M}/G\). This pushes forward to the desired homotopy, \(h\) on \(M\).

(iv)\(\ \Rightarrow\ \)(v)
For any given \(\tilde{g}\), the push forward \(h\) of \(\tilde{h}\) defined above is the desired homotopy.

(v)\(\ \Rightarrow\ \)(i)
Apply (v) to the pullback of \(g\) and push forward the resulting homotopy to \(M\).
\end{proof}

The question of whether the conditions of \Cref{lem:const_neg} are satisfied are not easy to check but the lemma affords us with several possible approaches to the problem. In this section we prove  \Cref{thm:intg_const_curv}, which gives a \emph{sufficient} condition for when \((M, g)\) admits a metric of constant, negative sectional curvature.

Before we can state it, let us agree on some notation and conventions. Given a metric \(g\) with Levi-Civita connection $\nabla$ on a manifold $M$, our conventions for the curvature tensor are
\[
\begin{split}
\Rm(X, Y) Z &= \nabla_X \nabla_Y Z - \nabla_Y \nabla_X Z - \nabla_{[X, Y]} Z, \\
\Rm(X, Y, Z, W) &= g(\Rm(X, Y) Z, W).
\end{split}
\]
Then the Ricci and scalar curvature are defined by
\eq{
\Ric(X, Y)&= \Tr \Rm(\cdot, X) Y,\\
		\Sc &=\Tr_{g}\Ric,
}
where $\Tr$ is the trace of an endomorphism and $\Tr_{g}$ is the trace of a bilinear form with respect to the metric $g$.

We define the Einstein tensor by
\eq{\Ein=\Ric-\frac{\Sc}{2}g.}

We may write the Ricci decomposition of the curvature tensor in three dimensions in the form
\begin{equation}
\label{eq:ricci_decomp}
\Rm = -\Ein \owedge g + \frac{\Tr \opEin}{2} g \owedge g
\end{equation}
where \(\owedge\) denotes the Kulkarni-Nomizu product. The sectional curvatures are
\[
K(X \wedge Y) = \frac{\Rm(X, Y, Y, X)}{\abs{X \wedge Y}^2}.
\]
The curvature operator \(\opRm\) is defined by
\[
\Rm(X, Y, Z, W) = \Rm(X \wedge Y, W \wedge Z) = g(\opRm(X \wedge Y), W \wedge Z).
\]

Then from the Ricci decomposition, given an orthonormal basis of eigenvectors \(E_i\) for \(\Ein\) with eigenvalues \(\lambda_i\) we have
\[
\begin{split}
g(\opRm(E_i \wedge E_j), E_p \wedge E_q) &= \left(-\Ein \owedge g + \frac{\Tr \opEin}{2} g \owedge g\right) (E_i \wedge E_j, E_p \wedge E_q) \\
&= g(\lambda_k E_i \wedge E_j, E_p \wedge E_q).
\end{split}
\]
Thus
\[
\opRm(E_i \wedge E_j) = \lambda_k E_i \wedge E_j
\]
and the eigenvalues of \(\opRm\) are precisely the eigenvalues of \(\Ein\). The sectional curvatures are then
\begin{equation}
\label{eq:sectional}
K (E_i \wedge E_j) = -\frac{g(\opRm(E_i \wedge E_j), E_j \wedge E_i)}{\abs{E_i \wedge E_j}^2} = -\lambda_k.
\end{equation}

Therefore \(\Ein\) is positive definite (respectively negatively definite) if and only if the sectional curvatures are negative (respectively positive). In the case of negative sectional curvature, \(\Ein\) is hence a metric. Writing \(\Ein(X, Y) = g(\opEin(X), Y)\), define \(\Ein^{-1}(X, Y) = g(\opEin^{-1} (X), Y)\).

Now we have the following theorem. We say that a symmetric \((0,2)\)-tensor \(T\) is \emph{Codazzi} if the covariant three-tensor \(\nabla T\) is totally symmetric.

\begin{thm}[Integrability and constant negative sectional curvature]
\label{thm:intg_const_curv}

Let \((M, g)\) be a closed Riemannian three-manifold of strictly negative sectional curvature with the integrability condition that the tensor \(\Ob = \sqrt{\det \opEin} \Ein^{-1}\) is \emph{Codazzi}. Then \(g\) is smoothly homotopic to a metric of constant, negative sectional curvature and hence in particular, \(M\) admits a metric of constant, negative sectional curvature.
\end{thm}

Note that as a special case of the considerably more general Thurston Geometrisation of three manifolds, already any closed three-manifold admitting a metric of negative sectional curvature also admits a metric on constant negative sectional curvature. The famous final resolution of Thurston Geometrisation by Perelman requires considerable machinery so the more direct and simpler proof described here in this case is desirable. Moreover, the homotopy in the theorem is not a consequence of Geometrisation. See \Cref{sec:geometrisation} for a brief discussion on Geometrisation.

\Cref{thm:intg_const_curv} follows from the embeddability \Cref{thm:intg_embed} and \cite[Theorem 1.1]{MR3344442}, which says that \(N\) may be deformed to the one-sheeted hyperboloid at infinity by the Gauss curvature flow.
Before we can state and prove \Cref{thm:intg_embed}, we need some more notation concerning extrinsic geometry.

Let \(\inpr{\cdot}{\cdot}\) denote the inner-product on Minkowski space and \(D\) the corresponding Levi-Civita connection. For a spacelike immersion \(F\cn M^n \to \R^{n,1}\) with \(M\) oriented, we define the second fundamental form $A$ with respect to a timelike, unit normal field $\nu$ by
\[
D_{F_{\ast} X} F_{\ast} Y = F_{\ast} \nabla_X Y + A(X,Y)\nu.
\]
We also define the Weingarten map via
\[
A(X, Y) =  g(\W(X), Y)
\]
and write $H = \Tr_{g}A = \Tr \W$.

The basic equations of hypersurfaces (Gauss equation) in Minkowski space are
\begin{equation}
\label{eq:gauss}
\begin{split}
\Rm(X, Y) Z &= A(X, Z) \W(Y) - A(Y, Z) \W(X), \\
\Ric(X, Y) &= g(\W^2(X) - H \W(X), Y), \\
\Sc &= \|A\|^2 - H^2.
\end{split}
\end{equation}
We can also relate the eigenvalues $\lambda_{i}$ of $\opEin$ with the principal curvatures $\kappa_{i}$ of the embedding. Namely, for distinct indices \(i,j,k\) we calculate with the help of \eqref{eq:gauss},
\eq{
\lambda_k & = \kappa_k^2 - \kappa_k \sum_l \kappa_l - \frac{\sum_l \kappa_l^2 - (\sum_l \kappa_l)^2}{2} \\
&= \kappa_i \kappa_j.
}

Hence there holds
\eq{\label{lem:ein_W}
\W = \sqrt{\det \opEin} \opEin^{-1},
}
since these endomorphisms are simultaneously diagonalizable and share the same eigenvalues.

Therefore, \(\Ein > 0\) if and only if all the principal curvatures \(\kappa_i\) have the same sign (which may take to be positive by swapping \(\nu\) with \(-\nu\) if necessary). That is, \(g\) has negative sectional curvature if and only if \(\Ein > 0\) if and only if \(F(\tilde{M})\) is a locally convex, co-compact, spacelike hypersurface.

\begin{rem}
We see a strong rigidity statement that the \emph{extrinsic geometry} of embedded, spacelike hypersurfaces is completely determined by the \emph{intrinsic geometry}. The extrinsic condition of local convexity is equivalent to the intrinsic condition of negative sectional curvature.
\end{rem}

Now we may give an intrinsic characterisation of when \((\tilde{M}, \tilde{g})\) embeds isometrically into Minkowski space as precisely when \(\Ob\) is Codazzi.

\begin{thm}[Integrability implies isometric embeddability]
\label{thm:intg_embed}

Let \((M, g)\) be a closed Riemannian three-manifold of strictly negative sectional curvature. Then the tensor \(\Ob = \sqrt{\det \opEin} \Ein^{-1}\) is Codazzi if and only if the Riemannian universal cover \((\tilde{M}, \tilde{g})\) embeds isometrically into Minkowski space \(\R^{3,1}\) as a locally convex, co-compact, spacelike hypersurface.
\end{thm}

\begin{proof}
First suppose \((\tilde{M}, \tilde{g})\) embeds isometrically into \(\R^{3,1}\). Since Minkowski space is flat, the second fundamental form \(A\) is Codazzi. Then \Cref{lem:ein_W} gives \(A = \Ob\), hence \(\Ob\) is Codazzi.

Conversely, suppose \(\Ob\) is Codazzi. A simple direct computation diagonalising \(\Ein\) shows $\Ob$ solves the contracted Gauss equation (second equation in \eqref{eq:gauss}). The Ricci decomposition \eqref{eq:ricci_decomp} for \(n=3\) then implies the full Gauss equation (first equation in \eqref{eq:gauss}). But the Gauss and Codazzi equations are precisely the integrability conditions required to locally integrate the over-determined system
\begin{align*}
F^{\ast} \inpr{\cdot}{\cdot} &= g \\
A(F) &= \Ob
\end{align*}
for \(F\). See for example \cite[Theorem 7]{MR1713298} and for a similar argument \cite[Chapter VI.12, p. 146 and Theorem V, p.393]{MR1013365}.

Since \((\tilde{M}, \tilde{g})\) is the universal cover of \((M, g)\) with strictly negative sectional curvature, \(\tilde{M}\) is diffeomorphic to \(\R^3\) by the Cartan-Hadamard theorem and we can globally integrate to obtain \(F\).
\end{proof}

\begin{proof}
[Proof of \Cref{thm:intg_const_curv}]

By \Cref{thm:intg_embed} we may embed \((\tilde{M}, \tilde{g})\) into Minkowski space as a locally convex, co-compact, spacelike hypersurface. By \cite[Theorem 1.1]{MR3344442}, the rescaled Gauss curvature flow deforms \((\tilde{M}, \tilde{g})\) smoothly to the hyperboloid at infinity with constant negative sectional curvature. Thus the flow provides a smooth homotopy from \((\tilde{M}, \tilde{g})\) to \((\tilde{M}, \tilde{g}_{\H})\).

According to \cite[Section 12]{MR3344442} (see also \Cref{lem:xcf_gcf}), the induced metric \(\tilde{g}_t\) on \(\tilde{M}\) evolves by the Cross Curvature Flow introduced in \cite{MR2055396} (see also \Cref{lem:xcf_gcf} below):
\[
\begin{cases}
\partial_t \tilde{g}_t &= 2\det \opEin(\tilde{g}_{t}) E^{-1}(\tilde{g}_{t}) \\
\tilde{g}_0 &= \tilde{g}.
\end{cases}
\]
At the initial time, we have \(\varphi^{\ast} \tilde{g}_0 = \tilde{g_0}\) for every \(\varphi \in G\). Then given any \(\varphi \in G\), \(\bar{g}_t = \varphi^{\ast} \tilde{g}_t\) is also a solution to the Cross Curvature Flow with the same initial condition. Hence by uniqueness of solutions (\cite{MR2055396,MR2207496}, \Cref{thm:xcf_existence_uniqueness} and \Cref{subsec:xcf_existence_uniqueness} below), \(\tilde{g}_t = \bar{g}_t = \varphi^{\ast} \tilde{g}_t\) and the flow is invariant under the action of \(G\). \Cref{lem:const_neg} then gives the result.
\end{proof}

The following conjecture suggests the integrability assumption in \Cref{thm:intg_const_curv} could be dropped.

\begin{conj}[\cite{MR2055396}]
\label{conj:chow_hamilton}

The XCF deforms arbitrary negatively curved metrics to a hyperbolic metric.
\end{conj}

Evidence for this conjecture includes convergence in the integrable case from \Cref{thm:intg_const_curv}, asymptotic stability of the hyperbolic metric under XCF (\cite{MR2448593} and \Cref{thm:hyperbolic_stability} below), monotonicity of an integral quantity measuring the deviation from constant curvature (\cite{MR2055396} and \Cref{thm:hyperbolicity} below).

\section{The Cross Curvature Flow}
\label{sec:xcf}
\subsection{Definition And Basic Properties Of The Flow}
\label{subsec:xcf_defn}

Let \((M, g)\) be a closed, Riemannian manifold and define
\eq{\adj\Ein(X,Y)=g(X,\adj\opEin(Y)),}
where $\adj$ is the adjugate of an endomorphism. The Cross Curvature Flow (in short, XCF) is the evolution equation,
\begin{equation}
\label{eq:xcf}
\begin{cases}
\partial_t g_t  &= 2 \adj\Ein(g_t), \\
g_0 &= g
\end{cases}
\end{equation}
where \(g_0\) has negative sectional curvature. When \(g_0\) has positive sectional curvature, we take instead \(\partial_t g = -2\adj\Ein\) though in this article we will not be concerned with this case.
\begin{rem}
Let \(\pi : \tilde{M} \to M\) be the universal cover, and \(\tilde{g}_t = \pi^{\ast} g_t\). Similarly to the proof of \Cref{thm:intg_const_curv}, we then have
\[
\partial_t \tilde{g}_t = \pi^{\ast} \partial_t \tilde{g}_t = \pi^{\ast} \adj\Ein(g_t) = \adj\Ein(\tilde{g}_t)
\]
and \(\tilde{g}_t\) solves the XCF with initial condition \(\tilde{g}_0 = \pi^{\ast} g_0\). Conversely, if \(\tilde{g}_t\) is a \(G\)-invariant solution of the XCF on \(\tilde{M}\), then there is a unique solution, \(g_t\) of the XCF on \(M\) such that \(\tilde{g}_t = \pi^{\ast} g_t\).
\end{rem}

The definition here makes sense in any dimension. If \(\opEin\) is invertible, we may also write
\[
\adj\Ein = \det \opEin \Ein^{-1} = \det\opEin g(\opEin^{-1} \cdot, \cdot).
\]
In three dimensions, \(g_t\) has negative sectional curvature if and only if \(\Ein\) is positive definite (hence \(\opEin\) is invertible). Then in an orthonormal basis of eigenvectors \(E_1, E_2, E_3\) for \(\opEin\), with eigenvalues \(\lambda_1, \lambda_2, \lambda_3\), we have for distinct indices, \(i, k, \ell\),
\[
\adj\opEin (E_i) = \det\opEin \opEin^{-1}(E_i) = \lambda_i \lambda_k \lambda_{\ell} \frac{1}{\lambda_i} E_i = \lambda_k \lambda_{\ell} E_i
\]
where \(i,k,\ell\) are distinct indices. Thus
\begin{equation}
\label{eq:cross_curvature}
\adj\Ein(E_i, E_j) = g(\lambda_k \lambda_{\ell} E_i, E_j) = \lambda_k \lambda_{\ell} \delta_{ij}.
\end{equation}
The tensor \(\adj\Ein\) is referred to as the \emph{cross curvature tensor}. The origin of the name is that the \(i\)'th eigenvalue of \(\adj\Ein\) is the ``cross term'' \(\lambda_k \lambda_{\ell}\) of the remaining eigenvalues.

There is an equivalent way to write \(\adj\Ein\) in three dimensions. In fact, both these definitions make sense in any dimension, however it is only in three dimensions that they coincide.
\begin{lemma}
\label{lem:xcf_equiv}

In three dimensions, we have
\[
\adj\Ein(X, Y) = -\frac{1}{2} \Ric_{\Ein} (X, Y) :=- \frac{1}{2} \Tr \br{Z \mapsto \Rm(\opEin(Z), X) Y}
\]
\end{lemma}
\begin{proof}
As noted above,
\[
\adj\Ein(E_i, E_j) = \lambda_k \lambda_{\ell} \delta_{ij}.
\]
There holds
\[
\begin{split}
\Ric_{\Ein}(E_i, E_j) &= \sum_{m=1}^3 \Rm(\opEin(E_m), E_i, E_j, E_m) \\
&= \sum_{m=1}^3 \lambda_m \Rm(E_m, E_i, E_j, E_m) \\
&= -\sum_{m=1}^3 \lambda_m \hat{\lambda}_{mi}\delta_{ij}.
\end{split}
\]
where \(\hat{\lambda}_{mi} = \lambda_k\) if \(m,i,k\) are distinct indices and is zero if \(m=i\). Now, if \(i = j\), and \(i, k, \ell\) are distinct indices, the sum is over \(m=k, \ell\) giving
\[
\Ric_{\Ein}(E_i, E_j) = -\left(\lambda_k \lambda_{\ell} + \lambda_{\ell} \lambda_k\right) = -2\lambda_k \lambda_{\ell}.
\]
Hence
\[
\frac{1}{2} \Ric_{\Ein}(E_i, E_j) =  -\lambda_k \lambda_{\ell} \delta_{ij} = -\adj\Ein(E_i, E_j).
\]
\end{proof}

\begin{rem}
In \cite[Lemma 3]{MR2055396} and \cite[Equation (3)]{MR2207496}, essentially the same result is obtained by contracting with the measure \(\mu\).
\end{rem}

In the case of integrable (and hence isometrically embeddable) solutions of XCF, we have the following observation of Ben Andrews.

\begin{lemma}[{\cite[Section 12]{MR3344442}}]
\label{lem:xcf_gcf}
The induced metric under the Gauss Curvature Flow of convex, spacelike, co-compact hypersurfaces in Minkowski space evolves by XCF.
\end{lemma}

\begin{proof}
The Gauss Curvature Flow of hypersurfaces in Minkowksi space is the evolution equation
\[
\partial_t F = K\nu
\]
where \(K = \det \W\) is the Gauss curvature. Under this equation the metric evolves by
\[
\partial_t g = 2KA.
\]
\Cref{lem:ein_W} gives \(A = \sqrt{\det \opEin} \Ein^{-1}\) and
\[
K = \det \W = \det (\sqrt{\det \opEin} \opEin^{-1}) = (\det \opEin)^{3/2} (\det \opEin)^{-1} = \sqrt{\det \opEin}
\]
so that
\[
\partial_t g = 2KA = 2 \det \opEin \Ein^{-1} = 2 \adj \Ein.
\]
\end{proof}

\subsection{Short Time Existence And Uniqueness}
\label{subsec:xcf_existence_uniqueness}

The question of short time existence and uniqueness of solutions to geometric, parabolic equations on tensor bundles is complicated by the diffeomorphism invariance of the problem leading to degeneracies in the principal symbol. For the Ricci flow, DeTurck described a method to deal with this degeneracy by breaking the diffeomorphism invariance (i.e. fixing a gauge) to obtain an equivalent, strictly parabolic flow referred to as the DeTurck flow \cite{MR697987}. In \cite[Section 6]{MR1375255} a further simplification of DeTurck's method was given. Buckland then adapted this approach to the XCF, while also pointing out a gap in the original proof of short time existence and uniqueness for XCF \cite{MR2207496} (see \Cref{rem:xcf_rf} below).

\begin{thm}[{\cite[Lemma 4]{MR2055396}}; {\cite[Theorem 1]{MR2207496}}]
\label{thm:xcf_existence_uniqueness}
Given any initial smooth metric \(g\) of negative sectional curvature on a closed three-manifold \(N\), there exists a unique solution \(g_t\) to the XCF,
\[
\begin{cases}
\partial_t g_t &= 2\adj\Ein(g_t) \\
g_0 &= g.
\end{cases}
\]
defined on a maximal time interval \([0, T)\) for some \(T > 0\) or \(T = \infty\).
\end{thm}

To prove the theorem, we need the principal symbol of the cross curvature tensor. Recall that for a nonlinear, second order differential operator \(D : E \to F\) acting on vector bundles \(E, F\) over \(M\), we define
\begin{equation}
\label{eq:symbol}
\sigma_{\xi} [D_s] = \sigma_{\xi} [D'_s]
\end{equation}
where \(s \in \Gamma(E)\) is a section of \(E\) and
\[
D'_s (u) = \partial_w|_{w=0} D(s + w u)
\]
is the linearisation of \(D\) around \(s\) acting on sections \(u \in \Gamma(E)\). Recall that the principal symbol, \(\sigma_{\xi} [D'_s]\) of \(D'_s\) is obtained by replacing second order derivatives in \(D'_s (u)\) by components of \(\xi\).

Recall that \(\Ein\) is positive definite when \(g\) has negative sectional curvature, hence \(\Ein\) is itself a metric. We may thus raise and lower indices using \(\Ein\) as well as take traces with \(\Ein\). However, rather than using \(\Ein\) directly to raise indices we use the metric raising of \(\Ein^{\sharp}\) defined on one-forms by \(\Ein^{\sharp} (\alpha, \beta) = E(\alpha^{\sharp}, \beta^{\sharp})\). The symbol may be conveniently computed using the equivalent formulation of the Cross Curvature tensor in \Cref{lem:xcf_equiv}.

\begin{lemma}[{\cite[Lemma 4]{MR2055396}; \cite[Theorem 1]{MR2207496}}]
\label{lem:xcf_symbol}
The principal symbol of the cross curvature tensor is,
\[
\sigma_{\xi} [\adj\Ein_g] (V) = \abs{\xi}^2_{\Ein}V - 2 \Sym \xi \otimes V(\sharp_{\Ein}\xi, \cdot) + \Tr_{\Ein} V \xi \otimes \xi.
\]
\end{lemma}

The degeneracy of the XCF is manifested in the second two terms. If they were absent, then the symbol would simply be \(V \mapsto \abs{\xi}^2_{\Ein} V\) which is clearly elliptic. However, the presence of the second two terms means that the symbol has a kernel and is hence not elliptic. This is what the DeTurck approach aims to address.

The proof of \Cref{thm:xcf_existence_uniqueness} makes us of the DeTurck XCF,
\begin{equation}
\label{eq:dtxcf}
\partial_t g = \adj\Ein (g) + \Lie_W g =: \dtxcf
\end{equation}
where
\[
W = g_0^{-1} \dive_g \left(g_0 - \frac{1}{2}\Tr_g g_0\right)^{\sharp}.
\]
Then the symbol is
\begin{equation}
\label{eq:dtxcf_symbol}
\sigma_{\xi} [\dtxcf_g] (V) = \abs{\xi}^2_{\Ein}V + 2 \Sym \xi \otimes V(\sharp \xi - \sharp_{\Ein}\xi, \cdot) + \left(\Tr_{\Ein} V - \Tr_g V\right) \xi \otimes \xi.
\end{equation}
The computation of the symbol of \(\Lie_W g\) is well known and is the same as that used for the Ricci flow. See \cite{MR2207496}, \cite[Sections 3.3, 3.4]{MR2061425} or \cite[Chapter 5]{MR2265040}.

That \(\sigma_{\xi}[\dtxcf_g]\) is uniformly elliptic is not readily apparent, but this is what Buckland shows \cite{MR2207496}. Existence now follows easily since the DeTurck XCF \eqref{eq:dtxcf} is uniformly parabolic hence unique solutions exist given any smooth initial data. Then if \(\bar{g}\) denotes the solution of the DeTurck XCF, \(g = \varphi_t^{\ast} \bar{g}\) solves the XCF where \(\varphi_t\) is the flow of \(W\).

For uniqueness, more work is required, but note that the vector field \(W\) is the same as that used for Ricci flow. Thus the same proof as for Ricci flow applies and we may appeal to \cite[Section 6]{MR1375255}. In the formulation used here, \(\varphi_t\) solves the harmonic map heat flow \((M, \bar{g}) \to (M, g_0)\). See \cite[Section 5.2]{MR2265040} and \cite[Sections 3.3, 3.4]{MR2061425} for details.

\begin{rem}
\label{rem:xcf_rf}

For the Ricci flow, the DeTurck Ricci flow is defined to be
\[
\partial_t g = -2\Ric (g) + \Lie_{W} g =: \dtrf (g)
\]
with the same \(W\) as \eqref{eq:dtxcf}. In fact, \(W\) is adapted to the Ricci flow rather than the XCF since the symbol for the Ricci flow is
\[
\sigma_{\xi} [-2\Ric_g] (V) = \abs{\xi}^2_{g}V - 2 \Sym \xi \otimes V(\sharp \xi, \cdot) + \Tr_g V \xi \otimes \xi.
\]
Then for the DeTurck Ricci flow, the last two terms cancel and the symbol for the DeTurck Ricci flow becomes
\[
\sigma_{\xi} [\dtrf_g] (V) = \abs{\xi}_g^2 V
\]
which is clearly uniformly elliptic. The choice of \(W\) comes from the fact \(\dive_g \Ein = 0\) so that defining \(L_g(T) = \dive_g(T - \tfrac{1}{2} \Tr_g T)\), we have the integrability condition
\[
L_g(\Ric(g)) = \dive_g \Ein(g) = 0.
\]
Here \(g \mapsto L_g\) depends on \(g\) to first order, and computing the symbol of \(g \mapsto L_g(\Ric(g)\) also shows the degeneracy in the operator \(\Ric\). The original proof of short time existence and uniqueness for Ricci flow made use of this fact employing the Nash-Moser implicit function theorem. See \cite[Sections 4-6]{Hamilton:/1982}.

One might try a similar approach for XCF, defining an appropriate \(W\) to cancel terms so that the symbol becomes \(V \mapsto |\xi|_{\Ein}^2 V\). Equation \eqref{eq:xcf_integ} below suggests the choice
\[
W = g_0^{-1}\left(\Ein^{ij} \nabla_i \adj \Grav(g_0)_{jk} - \frac{1}{2} \Ein^{ij} \nabla_k \adj\Grav(g_0)_{ij}\right)^{\sharp}
\]
where \(\Grav(g_0) = g_0 - \tfrac{1}{2} \Tr_g g_0\) is the Einstein gravity tensor. Presumably such an approach works, with the appropriate replacement for the harmonic map heat flow to obtain uniqueness. Note this equation also leads to the integrability equation
\[
L_g(\adj\Ein(g)) = 0
\]
where \(L_g(T) = \Ein^{ij} \nabla_i \adj \Grav(T)_{jk} - \frac{1}{2} \Ein^{ij} \nabla_k \adj\Grav(T)_{ij}\).

\cite{MR2207496} observed that \(g \mapsto L_g\) is second order in \(g\) so that the argument in \cite[Sections 4-6]{Hamilton:/1982} cannot be applied. However, as we have seen, even though the terms in \eqref{eq:dtxcf_symbol} do not cancel, the DeTurck flow using \(W\) from the Ricci flow is still uniformly parabolic. Existence for XCF follows immediately, and uniqueness is obtained by the known result using the harmonic map heat flow as in the Ricci flow.
\end{rem}

\subsection{Basic Identities And Evolution Equations}
\label{subsec:xcf_identities}

Now we give some fundamental identities and evolution equations for the study of the XCF. Since there are many traces and swapping slots in what follows, it is convenient to use index notation. For convenience, we define \(V = \Ein^{-1}\) (that is \(V(X, Y) = g(\opEin^{-1} (X), Y)\)).

The first basic identity is an analogue for the Cross Curvature tensor of the fact that according to the contracted second Bianchi identity, \(\dive_g \Ein = 0\). It is fundamental in deriving various identities and is closely related to short time existence and uniqueness as discussed above. According to \cite[Lemma 1 (b)]{MR2055396},
\begin{equation}
\label{eq:xcf_integ}
\Ein^{ij} \nabla_i \adj\Ein_{jk} = \frac{1}{2} \Ein^{ij} \nabla_k \adj\Ein_{ij}.
\end{equation}

The general, non integrable case poses a number of difficulties. At the heart of these difficulties is the \emph{Devil tensor} \eqref{eq:devil} that vanishes if and only if the solution is integrable (\Cref{lem:cubicform_codazzi}). Let us define
\begin{equation}
\label{eq:T}
\T^{kij} = \Ein^{kl}\nabla_l \Ein^{ij}, \quad \T^i = V_{jk}\T^{ijk} = \Ein^{ij}\nabla_j \ln\det\opEin.
\end{equation}
From the irreducible decomposition of $\T^{ijk}$ under the action of \(O(3)\) via the metric \(V\) we have (\cite[p. 6]{MR2055396}),
\[
\T^{ijk}-\T^{jik} = \mathrm{L}^{ijk}-\mathrm{L}^{jik} + \frac{1}{2}\left(\T^i\Ein^{jk}-\T^j\Ein^{ik}\right)
\]
where the traces of $\mathrm{L}^{ijk}$ with respect to $V$ are zero. The Devil tensor is defined by
\begin{equation}
\label{eq:devil}
\Dv^{ijk}=\mathrm{L}^{ijk}-\mathrm{L}^{jik} = \T^{ijk}-\T^{jik} - \frac{1}{2}\left(\T^i\Ein^{jk}-\T^j\Ein^{ik}\right).
\end{equation}
It satisfies the curvature-like identities,
\[
\Dv^{ijk} = -\Dv^{jik},\quad \Dv^{ijk}+\Dv^{kij}+\Dv^{jki} = 0,
\]
and trace identities,
\begin{align*}
V_{ij}\Dv^{ijk}=&V_{ik}\Dv^{ijk}=V_{jk}\Dv^{ijk}=0,\\
V_{ij}\nabla_k\Dv^{kij}=&\frac{1}{2}|\Dv^{ijk}|_V^2.
\end{align*}

\begin{lemma}
\label{lem:cubicform_codazzi}

We have the following identity.
\[
|\Dv^{ijk}|_V^2=\frac{1}{\det\opEin}|\nabla_i\Ob_{jk}-\nabla_j\Ob_{ik}|_{\Ein}^2.
\]
In particular, \(\Dv\equiv 0\) if and only if \(\Ob\) is Codazzi.
\end{lemma}

\begin{proof}
Using the definition we calculate
\begin{align*}
\T^{kij}=&-\Ein^{kl}\Ein^{im}\Ein^{jn}\nabla_l V_{mn} = -\Ein^{kl}\Ein^{im}\Ein^{jn}\nabla_l \left(\frac{1}{\sqrt{\det\opEin}} \Ob_{mn}\right)\\
=&-\frac{1}{\sqrt{\det\opEin}}\Ein^{kl}\Ein^{im}\Ein^{jn}\nabla_l \Ob_{mn}+\frac{1}{2}\T^k\Ein^{ij}.
\end{align*}
Thus we obtain
\begin{align*}
\T^{kij}-\T^{ikj}=&\frac{\Ein^{kl}\Ein^{im}\Ein^{jn}}{\sqrt{\det\opEin}}\left(\nabla_m \Ob_{ln}-\nabla_l \Ob_{mn}\right)+\frac{1}{2}(\T^k\Ein^{ij}-\T^i\Ein^{kj}).
\end{align*}
Then by definition,
\[
\Dv^{kij} = \frac{\Ein^{kl}\Ein^{im}\Ein^{jn}}{\sqrt{\det\opEin}}\left(\nabla_m \Ob_{ln}-\nabla_l \Ob_{mn}\right)
\]
and hence
\[
|\Dv^{ijk}|_V^2\det\opEin = |\nabla_i\Ob_{jk}-\nabla_j\Ob_{ik}|_{\Ein}^2.
\]
\end{proof}

To round out this section, let us point out some evolution equations. Define the operator,
\begin{equation}
\label{eq:box}
\Box = \Tr_{\Ein^{\sharp}} \nabla^2 = \Ein^{ij} \nabla^2_{ij}.
\end{equation}
It is uniformly elliptic for each fixed \(t\) on any time interval \([0, \tau]\) on which \(\Ein > 0\).

We have
\begin{equation}
\label{eq:dtEin}
\partial_t \Ein^{ij} = \Box \Ein^{ij} - \nabla_{\ell} \Ein^{ki} \nabla_k \Ein^{\ell j} - 4 \det \opEin g^{ij}.
\end{equation}
and
\begin{equation}
\label{eq:dtdetEin}
\partial_t \det \opEin = \Box \det \opEin -\left(\frac{1}{2(\det\opEin)^2} \abs{\Ein^{ij} \nabla_j \det \opEin}^2 - \frac{1}{2}\abs{\mathrm{D}^{ijk}}_V^2 + 2\mathrm{H}\right)\det \opEin.
\end{equation}
Here $\mathrm{H}=\Tr_g\adj\Ein.$ 
Equation \eqref{eq:dtEin} is obtained by combining equation \eqref{eq:xcf_integ} (\cite[Lemma 1(a)]{MR2055396}) with \cite[Lemma 5]{MR2055396} giving the evolution of \(\Ein\). Equation \eqref{eq:dtdetEin} is then obtained from \eqref{eq:dtEin} using \eqref{eq:T}.

The maximum principle cannot be immediately applied to ensure positivity of \(\Ein\) (and hence negative sectional curvature) is preserved since the last term of each equation has the wrong sign. However, a slightly less direct approach using the evolution of \(\det \opEin\) effectively leads to the desired conclusion (\Cref{prop:negative_perserved}).

The XCF expands negatively curved metrics as can be seen from the evolution of the volume:
\begin{equation}
\label{eq:dtvol}
\frac{d}{dt} \operatorname{Vol}_g (M) = \int_M \mathrm{H}d\mu > 0.
\end{equation}

Such expansion is also apparent in the evolution of hyperbolic metrics.
\begin{example}
\label{eq:hyperbolic}

Let \(g_0\) be a hyperbolic metric with constant sectional curvature \(K_0 < 0\). Then the solution of the XCF is
\[
g(t) = \sqrt{4K_0^2 t + 1} g_0
\]
since scaling \(g_0\) by \(r\) results in \(\adj\Ein(g_t) = \frac{1}{r} \adj\Ein(g_0) = \frac{K_0^2}{r} g_0\).
\end{example}

We see that up to scale, hyperbolic metrics are fixed points. In fact, they are the only fixed points up to scale and reparametrisation (i.e. solitons). See \Cref{subsec:xcf_harnack_solitons} below.

\subsection{Towards Hyperbolic Convergence}
\label{subsec:xcf_hyperbolic_convergence}

The following results support \Cref{conj:chow_hamilton} that the XCF deforms arbitrary negatively curved metrics to a hyperbolic metric. Under the XCF, it is not so easy to prove directly that negative sectional curvature is preserved. However, it is possible to prove that if \(\det \opEin \to 0\), then a singularity must occur and thus negative curvature is effectively preserved since that corresponds to positive of \(\Ein\).

\begin{prop}[{\cite[p. 8]{MR2055396}}]
\label{prop:negative_perserved}

Let $T$ be the maximal time of smooth existence of XCF. Then \(\det \opEin > 0\) on \([0, T)\) and hence the sectional curvatures remain negative as long as the solution is smooth.
\end{prop}

The proof follows from equation \eqref{eq:dtdetEin} by writing,
\[
\partial_t \ln \det \opEin = \Box \ln \det \opEin + \frac{1}{2}\abs{\mathrm{T}^{ijk} - \mathrm{T}^{jik}}_V^2- 2 \mathrm{H}.
\]
The last term has the wrong sign to apply the maximum principle. However, \(\partial_t \min \ln \det\opEin \geq -2 \mathrm{H}\) at the minimum of \(\ln \det\opEin\). If $t_0<T$ and \(\ln\det\opEin \to -\infty\) as $t\to t_0$, then \(\mathrm{H} \to \infty\). In other words one eigenvalue of \(\opEin\) goes to zero while another goes to \(\infty\) as $t$ approaches $t_0$.

Going the other way, namely proving that if \([0, T)\) is the maximal time of existence with \(T < \infty\), then necessarily \(\inf_{t \in [0, T)} \det \opEin = 0\) would show that finite time existence implies blow up of \(\Ein\). Such a result would serve as a proxy for more general estimates, such as the smoothing estimates enjoyed by the Ricci flow. The lack of such estimates (even \(C^2\) estimates) is perhaps the major outstanding issue for the XCF. In particular, it could be that the solution blows up in finite time, yet \(\det \opEin\) has a positive lower bound.

The following monotonicity property also supports \Cref{conj:chow_hamilton} that the flow evolves negative curvature metrics to a hyperbolic metric.

\begin{thm}[{\cite[Theorem 8]{MR2055396}}]
\label{thm:hyperbolicity}
Under the XCF
\[J(M_t):=\int \frac{1}{3} \Tr \opEin - (\det\opEin)^{\frac{1}{3}}d\mu\]
is non-increasing. Moreover, \(\frac{d}{dt}J(M_t) = 0\) if and only if $g$ has constant curvature.
\end{thm}

The theorem follows by direct computation and noting that by the arithmetic-geometric mean inequality applied to the eigenvalues of \(\opEin\), \(\tfrac{1}{3} \Tr\opEin \geq (\det \opEin)^{1/3}\) with equality if and only if \(\Ein = \lambda g\) for some smooth function \(\lambda\). In the equality case, since \(\dive_g \Ein = \dive_g g = 0\), we must then also have \(\lambda \equiv \text{ const.}\) and hence equality occurs if and only if \(g\) has constant sectional curvature equal to \(-\lambda\).

To finish this section, we consider the stability of hyperbolic metrics under the flow. Stability is necessary for \Cref{conj:chow_hamilton} to hold, and indeed this is the case. To state the result, some normalisation is necessary. Typically normalising to fix volume is the approach taken, but here following \cite{MR2448593} a different normalisation is chosen. For any \(K < 0\), note that a constant curvature metric \(g_K\) of sectional curvature \(K\) satisfies
\[
\adj\Ein = K^2 g
\]
by \eqref{eq:sectional} and \eqref{eq:cross_curvature}. Then the flow
\[
\partial_t g = 2\adj\Ein(g) - 2K^2 g
\]
has \(g_K\) as a fixed point. Note that up to scaling and diffeomorphism, solutions of this flow are equivalent to solutions of the XCF by \cite[Lemma 1]{MR2448593}.

The linearisation of the corresponding DeTurck flow around a hyperbolic metric is (\cite[Lemma 2]{MR2448593})
\[
\partial_t V = -K \Delta V -2K^2 \Tr_g V g + 2K^2 V.
\]
Then the spectrum of the right hand side is contained in \((-\infty, -1]\) (\cite[Section 5]{MR2448593}), from which the following theorem follows by standard stable manifold theory.

\begin{thm}[{\cite[Theorem 4]{MR2448593}}]
\label{thm:hyperbolic_stability}

Any constant curvature metric is asymptotically stable under the XCF after suitable normalisation.
\end{thm}

\subsection{Harnack inequality And Solitons}
\label{subsec:xcf_harnack_solitons}

Differential Li-Yau-Hamilton Harnack inequalities have proved an indispensable tool in the study of curvature flows. For the XCF, \cite[p. 9]{MR2055396} remarked that it is hoped such an inequality holds. For integrable solutions of the XCF this is true by proving the analogous Harnack inequality for the corresponding solution of the embedded Gauss curvature flow. In general, the non-vanishing of the Devil Tensor \eqref{eq:devil} causes significant difficulties in obtaining a Harnack inequality. The evolution of the devil tensor is very complicated and it's not clear whether or not it is amenable to the  maximum principle.

\begin{thm}(\cite[Section 6]{BIS4})
\label{thm:harnack}
The following Li-Yau-Hamilton Harnack inequality holds for integrable solutions to the XCF,
\[
\partial_t \sqrt{\det\opEin} - \frac{1}{\sqrt{\det\opEin}} \abs{\nabla \sqrt{\det\opEin}}_{\Ein}^2 + \frac{3}{4t}\sqrt{\det\opEin} \geq 0.
\]
\end{thm}

Solitons are closely related to the Harnack inequality. They are fixed points of the flow modulo scaling and diffeomorphism. As such, they are the expected limits (up to rescaling) of the flow. Unlike other flows such as the Ricci Flow, solitons for the XCF are very rigid. In fact something stronger is true. Recall that a soliton is a solution of XCF such that \(g_t = \lambda(t) \varphi_t^{\ast} g_0\) where \(\lambda\) is a positive function and \(\varphi_t\) a one-parameter family of diffeomorphisms. More generally, a \emph{breather} is a solution such that \(g_{t_0} = \lambda \varphi^{\ast} g_0\) for some \(t_0 ,\lambda > 0\) and a diffeomorphism \(\varphi\).

\begin{thm}[{\cite{MR2302600}}]
The only breathers, hence also solitons to the XCF are constant curvature metrics.
\end{thm}

This follows by monotonicity and scaling: Since volume is increasing by equation \eqref{eq:dtvol} and scales as \(\lambda^{3/2}\), for a breather we must have \(\lambda > 1\). On the other hand, since \(J\) scales as \(\lambda^{1/2}\), if \(J\) strictly decreases, then we obtain the contradiction \(\lambda < 1\). Therefore, \(J\) is constant and hence \(g_t\) has constant sectional curvature by the equality case of \Cref{thm:hyperbolicity}.

The theorem further supports \Cref{conj:chow_hamilton} in that if a limiting metric exists it should be stationary for the flow (up to rescaling and reparametrisation) and hence a soliton. Thus the only expected limits have constant curvature. A Harnack inequality for general solutions would further support the conjecture since it may be used to obtain improved smoothing estimates for the flow.

\subsection{Monotonicity Of Einstein Volume}
\label{subsec:xcf_volume}

Since the Einstein tensor is a metric, it induces a volume form. By \cite[Proposition 9]{MR2055396}, the Einstein volume is monotone non-decreasing along the XCF. We may strengthen this result, characterising solutions to the XCF such that \(\frac{d}{dt} I(M_t) = 0\) as precisely the integrable solutions, a question posed in \cite{MR2055396}, page 9.

\begin{thm}
\label{thm:volume_monotonicity}
Under the XCF of negative sectional curvature, the \emph{Einstein Volume},
\[
I(M_t):=\int \sqrt{\det\opEin}\,d\mu
\]
is non-decreasing. Moreover, \(\frac{d}{dt} I(M_t)= 0\) if and only if \(\Ob\) is Codazzi if and only if the Riemannian universal cover \((\tilde{M}, \tilde{g})\) embeds isometrically into Minkowski space \(\R^{3,1}\) as a locally convex, co-compact, spacelike hypersurface.
\end{thm}

\begin{proof}
By \cite[Proposition 9]{MR2055396}, we have
\[
(\partial_t - \Box) (\det\opEin)^\eta = \left(\frac{\eta}{2}|\Dv^{ijk}|_V^2 + \frac{\eta}{2}(1-2\eta)|\T^i|^2_V - 2\eta \mathrm{H}\right)(\det\opEin)^\eta,
\]
where \(\Box\) is defined in \eqref{eq:box}. Integrating by parts gives the identity
\[
\frac{d}{dt}\int(\det\opEin)^\eta d\mu = \int \left(\frac{\eta}{2}|\Dv^{ijk}|_V^2 + \frac{\eta}{2}(1-2\eta)|\T^i|^2_V + (1-2\eta)\mathrm{H}\right)(\det\opEin)^\eta d\mu.
\]
In particular, for $\eta=\frac{1}{2}$ we obtain
\[
\frac{d}{dt}I(M_t) = \frac{1}{4}\int |\Dv^{ijk}|_V^2\sqrt{\det\opEin}d\mu
\]
and monotonicity follows. Now apply \Cref{lem:cubicform_codazzi} and \Cref{thm:intg_embed} to obtain the second part.
\end{proof}

\section*{Acknowledgments}

Parts of this work were written while JS enjoyed the hospitality of the Department of Mathematics at Columbia University in New York, a visit which is funded by the "Deutsche Forschungsgemeinschaft" (DFG, German research foundation) within the research scholarship "Quermassintegral preserving local curvature flows", grant number SCHE 1879/3-1. JS would like to thank the DFG, Columbia University and especially Prof. Simon Brendle for their support. PB was supported by the ARC within the research grant “Analysis of fully non-linear geometric problems and differential equations”, number DE180100110.

\printbibliography

\end{document}